\documentclass{amsart}
\usepackage{amsmath, amssymb, amsthm}
\usepackage{url}
\usepackage{tikz}
\usepackage[utf8]{inputenc}
\usepackage[hidelinks]{hyperref}
\usepackage{mathtools, graphicx}
\usepackage{pgfplots}
\usepackage{mathrsfs}
\usepackage{upgreek}
\usetikzlibrary{arrows}
\pgfplotsset{compat=1.15}
\usepackage{babel}
\usepackage{microtype}
\numberwithin{equation}{section}
\theoremstyle{plain}
\newtheorem{theorem}{Theorem}[section]
\newtheorem{lemma}[theorem]{Lemma}

\newtheorem{corollary}[theorem]{Corollary}
\newtheorem{proposition}[theorem]{Proposition}

\newtheorem*{observation}{Observation}

\usepackage{etoolbox}

\begin{document}
	\title[Local asymptotics and estimates of $m_D$- metric]{Local asymptotics near a smooth boundary point and estimates for a hyperbolic-type metric}
	\date{\today}
	
	\author[P. Naskar]{Pritam Naskar}
	\address{Pritam Naskar\\ Department of Mathematics \\
		Indian Institute of Technology Indore \\
		Simrol,  Indore,  Madhya Pradesh 453552, India.} 
	\email{naskar.pritam2000@gmail.com, phd2201241008@iiti.ac.in }

	\maketitle

	\begin{abstract}
	  In this paper, we investigate the local boundary behaviour of a recently developed hyperbolic-type metric $m_D$. First, employing a boundary-flattening technique and local behaviour of $m_D$-geodesics, we establish its asymptotic formula near any $C^1$-smooth boundary point. Next, we introduce a metric quantity analogous to the Nikolov--Andreev metric and show that $m_D$ is the inner metric associated with it. Finally, by establishing a sharp two-sided comparison inequality, we obtain an improved lower bound for the $m_D$-metric.
    \end{abstract}
	\section{\bf Introduction}\label{intro}
	Intrinsic geometry and hyperbolic-type metrics on proper subdomains $D\subset\mathbb{R}^n$, $n\geq 2$, is an important area in modern geometric function theory. The most fundamental object is the Poincar\'{e} hyperbolic metric model on the unit ball or on the upper half-space. To study in a general domain having no conformal invariance, Gehring and Palka \cite{GP} introduced the \emph{quasihyperbolic metric} $k_D$ as 
	\begin{equation}\label{eq:defn of kmetric}
		k_D(x,y) := \inf_{\gamma \in \Gamma_{xy}} \int_{\gamma} \frac{|dz|}{\delta_D(z)},
	\end{equation}
	where for a point $z\in D$ the Euclidean distance from $z$ to the boundary $\partial D$ of $D$ is defined by $\delta_D(z) := \inf\{|z-\xi| : \xi \in \partial D\}$ and $\Gamma_{xy}$ denotes the family of all rectifiable paths $\gamma$ joining $x$ to $y$ in $D$. Later, Gehring and Osgood \cite{GO} systematically developed many geometric and distortion properties of $k_D$ along with the characterization of \emph{uniform domains} in terms of metric inequality.

	Since path-integrated metrics such as $k_D$ are generally difficult to study explicitly, several closed-form metrics were introduced to approximate them. Gehring and Osgood first introduced distance ratio metrics as a companion of $k_D$ in \cite{GO}, which was later redefined by Vuorinen as
	\begin{equation}\label{eq:defn of j}
		j_D(x,y) := \log\left(1 + \frac{|x-y|}{\delta_D(x) \wedge \delta_D(y)}\right).
	\end{equation}
	Recently, in \cite{MaNaSa 1, MaNaSa 2}, this framework has been extended by introducing a generalized hyperbolic-type metric $m_D$, defined by 
	\begin{equation}
		m_{D}\left(x,y\right):=\inf_{\gamma\in \Gamma_{xy}}\int_{\gamma}\frac{d(D)}{\eta_D(z)}|dz|,
	\end{equation}
	and a distance ratio analogue companion metric 
	\begin{equation}\label{eq:defn of zeta}
		\zeta_D(x,y):=\log\left(1+\frac{d(D)\,|x-y|}{\min\{\eta_D(x), \eta_D(y)\}}\right),
	\end{equation}
	respectively, where $d(D)$ denotes the Euclidean diameter of the domain $D$, which is infinite if $D$ is unbounded and $\eta_D(z):=\delta_D(z)(d(D)-\delta_D(z))$. 
	
	In the unbounded case, the density $d(D)/\eta_D(z)$ naturally coincides with the quasihyperbolic density. It is worth mentioning that the metric $m_D$ coincides with the hyperbolic metric in balls and half-spaces, whereas it agrees with the quasihyperbolic metric in any unbounded domain $D$.
	
	In general, it is non-trivial to provide a compact formula for such a path-integrated metric. Our primary goal is to provide a compact formula for the $m_D$-metric near a $C^1$-smooth boundary point. The intuition behind this idea is described in Figure \ref{motivation_fig}, where we complete the diagram by the explicit quantity
	\begin{align}\label{eq:psi}
		\psi_D(x,y) := 2\sinh^{-1}\frac{\text{diam}(D)|x-y|}{2\sqrt{\eta_D(x)\eta_D(y)}},
	\end{align}
	and $s_{\mathbb{D}}$, $s_{\mathbb{H}}$ denotes the compact formula for the hyperbolic metric $h_{\mathbb{D}}$ in the unit disk and $h_{\mathbb{H}}$ in the upper half-plane, respectively (see Exercise 1.8(iii) and Theorem 1.2.6 of \cite{Ka}). Note that, the function $\psi_D$ simultaneously generalizes both $s_{\mathbb{D}}$ and $s_{\mathbb{H}}$.
		\vspace*{-1.5cm}
		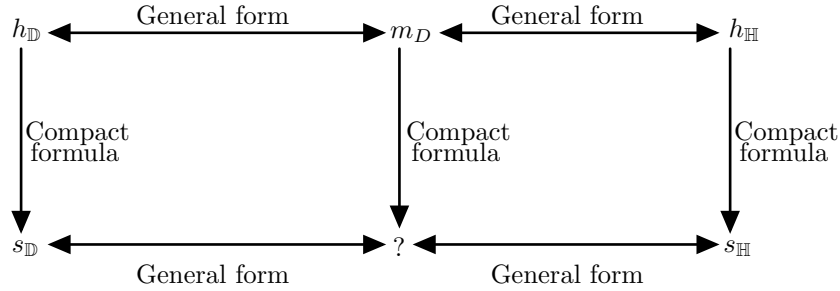
\begin{figure}[h]
		\centering
		\begin{tikzpicture}[line cap=round,line join=round,>=triangle 45,x=0.7cm,y=0.7cm]
			\clip(-6.5,-5) rectangle (12,5);
			\draw [<->,line width=1pt] (-3.7,2) -- (2.70,2);
			\draw [<->,line width=1pt] (-3.7,-2) -- (2.70,-2);
			\draw (-4.55,2.44) node[anchor=north west] {$h_{\mathbb{D}}$};
			\draw (2.60,2.30) node[anchor=north west] {$m_D$};
			\draw (-4.55,-1.75) node[anchor=north west] {$s_{\mathbb{D}}$};
			\draw (2.65,-1.7) node[anchor=north west] {$?$};
			\draw (-2.2,2.7) node[anchor=north west] {$\textrm{General form}$};
			\draw [->,line width=1pt] (-4.2,1.7) -- (-4.2,-1.8);
			\draw [->,line width=1pt] (2.95,1.7) -- (2.95,-1.7);
			\draw (2.9,0.5) node[anchor=north west] {$\textrm{Compact}$};
			\draw (3.0,0.1) node[anchor=north west] {$\textrm{formula}$};
			\draw (-4.3,0.5) node[anchor=north west] {$\textrm{Compact}$};
			\draw (-4.2,0.1) node[anchor=north west] {$\textrm{formula}$};
			\draw (-2.2,-2.2) node[anchor=north west] {$\textrm{General form}$};
			\draw (9,2.44) node[anchor=north west] {$h_{\mathbb{H}}$};
			\draw [<->,line width=1pt] (3.7,2) -- (9.00,2);
			\draw [<->,line width=1pt] (3.2,-2) -- (9.0,-2);
			\draw (8.9,-1.75) node[anchor=north west] {$s_{\mathbb{H}}$};
			\draw [->,line width=1pt] (9.2,1.7) -- (9.2,-1.8);
			\draw (9.1,0.5) node[anchor=north west] {$\textrm{Compact}$};
			\draw (9.2,0.1) node[anchor=north west] {$\textrm{formula}$};
			\draw (4.5,-2.2) node[anchor=north west] {$\textrm{General form}$};
			\draw (4.5,2.7) node[anchor=north west] {$\textrm{General form}$};
		\end{tikzpicture}
		\vspace*{-1.75cm}
		\caption{Motivation.}
		\label{motivation_fig}
	\end{figure}
	
	Secondly, motivated by the Nikolov--Andreev metric, we introduce the following quantity 
	\begin{align}\label{eq:i*}
		i_D^*(x,y) :=2\log\frac{\eta_{D}(x)+\eta_{D}(y)+d(D)|x-y|}{2\sqrt{\eta_{D}(x)\eta_{D}(y)}}
	\end{align}
	and prove that it is a metric, even if we replace $\eta_D$ by any positive Lipschitz function. We establish a comparison inequality with $\zeta_D$ and prove that $i_D^*$ possesses $m_D$ as its inner metric. Finally, applying these results, we obtain an improved lower bound of $m_D$ by the metric $i_D^*$. We refer the reader to \cite{Moc25, NiAn, NiTh} for related works.
	\section{\bf Preliminary results}\label{preli}
This section contains all the auxiliary results needed to prove the main conclusion of the manuscript.
	\begin{lemma}\label{lemma: uniform domain and m-geod}
		If $D$ is a uniform domain, then there exists a constant $b$ such that 
		\begin{align*}
			\ell(\gamma(x,y))\leq b\,|x-&y|,\\
			\min\left\{\ell(\gamma(x,w)),\ell(\gamma(w,y))\right\}&\leq b\, \delta_D(w)
		\end{align*}
		for each $m_D$-geodesic $\gamma$ in $D$ and each triplet of points $x,w,y$ on $\gamma$.
	\end{lemma}
	\begin{proof}
		The proof follows from the direct use of the equivalence of the quasihyperbolic and the $m_D$-metric \cite[Theorem 3.5]{MaNaSa 1} in place of \cite[eq.~(2.18)]{GO} and the consequent discussion in the same paper of Gehring and Osgood.
	\end{proof}
	\begin{observation}\label{Observation}
		\normalfont A domain with a $C^1$-smooth boundary point is locally a uniform domain near that boundary point. Consequently, $m_D$-geodesics remain localized within a small neighbourhood of their endpoints near such a boundary point (cf.\ Lemma~\ref{lemma: uniform domain and m-geod}). More precisely, for a pair of points $x,y$ in a sufficiently small neighbourhood $\mathcal{U} \subset D$ near $u \in \partial D$, there exists a smaller neighbourhood $\mathcal{V}$ such that any $m_D$-geodesic connecting $x$ and $y$ does not leave $\mathcal{U}$.
	\end{observation}
	The proof of the following useful result is readily available in \cite[Lemma 3.3(iii)]{MaNaSa 2}.
	\begin{lemma}\label{lemma:eta is lips}
		Let $x,y$ be any two points in a bounded domain $D$ in
		$\mathbb{R}^{n}$, $n\geq 2$. Then the function $\eta_D$ is $d(D)$-Lipschitz.
	\end{lemma}
	\begin{proposition}\label{prop:distortion-Lip}
		Let $f:\mathbb{R}^n\to \mathbb{R}^n$ be a Lipschitz mapping, that is, there exists a constant $L>0$ such that 
		$$
		|x-y|/L\leq |f(x)-f(y)|\leq L |x-y|.
		$$
		Then, for any proper subdomain $D\subset\mathbb{R}^n$, we have
		$$
        m_D(x,y)/L'\leq m_{f(D)}(f(x),f(y))\leq L' m_D(x,y),
        $$		
        where $L'=\min\{2L^2,L^3\}$.
	\end{proposition}
	\begin{proof}
		By definition of the $m_D$-metric, we have
		\begin{align*}
			m_{D}\left(x,y\right)
			=\inf_{\gamma\in \Gamma_{xy}}\int_{\gamma}\frac{d(D)}{\eta_D(z)}|dz|
			\leq\inf_{\gamma\in \Gamma_{f(x)(y)}}\int_{\gamma}L^3\frac{d(f(D))}{\eta_{f(D)}(z)}|df(z)|
			=L^3m_{f(D)}\left(f(x),f(y)\right).
		\end{align*}
		Applying the same argument to $f^{-1}$ yields $m_{f(D)}\left(f(x),f(y)\right) \leq L^3 m_{D}\left(x,y\right)$. 
		
		Alternatively, we use the equivalence between $m_D$ and $k_D$ \cite[Theorem 4.4]{MaNaSa 1} together with the bi-Lipschitz property of $k_D$ \cite[Exercise 6.8]{HaKlVu} under $f$ to conclude that 
		$$
		\frac{m_D(x,y)}{2L^2}\leq m_{f(D)}(f(x),f(y))\leq 2L^2 m_D(x,y).
		$$
		Combining both bounds yields $L' = \min\{2L^2, L^3\}$.
	\end{proof}
	
	The following standard inequality will be needed in the main section.
	\begin{lemma}\label{lem:sinh inequality}
		For $w\geq 0$ and a constant $c\in \mathbb{R}_{\geq 0}$, we have
      \begin{align*}
	\sinh^{-1}(cw)
	\begin{cases}
		\geq c\sinh^{-1}(w), & \text{if } 0 \leq c \leq 1 \\
		\leq c\sinh^{-1}(w), & \text{if } c \geq 1.
	\end{cases}
      \end{align*}
	\end{lemma}
	The proof of the result can easily be derived using basic calculus.
\section{\bf Main results}\label{main}
\subsection{Estimate of $m_D$ near a smooth boundary point}
In this subsection, we prove that near a $C^1$-smooth boundary point, the metric $m_D$ agrees with $\psi_D$ in the limiting sense. That is, 
	\begin{theorem}\label{thm:estimate near bdry}
		Let $u$ be a $\mathcal{C}^1$-boundary point of a proper subdomain in $\mathbb{R}^n$, $n\geq 2$. Then we have
		\begin{align}
			\lim_{\substack{x,y\to u\in \partial D \\ x\neq y}}\frac{m_D(x,y)}{\psi_D(x,y)}=1.
		\end{align}
		In other words, near a $\mathcal{C}^1$-smooth boundary point, the metric $m_D$ is given by the compact formula $\psi_D$.
	\end{theorem}
	\begin{proof}
		The strategy of the proof is to flatten the boundary $\partial D$ locally near $u$, where the $\mathcal{C}^1$-smoothness of $\partial D$ is assumed. We begin the proof with some elementary transformations so that the point maps to the origin and there exists a neighbourhood $\mathcal{U}$ of the origin where the boundary can be written as an implicit function due to the smoothness assumption. That is, due to our assumptions and the $\mathcal{C}^1$-smoothness at the origin, there exists a $\mathcal{C}^1$ function $\uptau(x)$ in $\mathbb{R}^{n-1}$ such that 
		$$
		\partial D=\left\{(x_1,x_2,\dots,x_n): x_1=-\uptau(x_2,x_3,\dots,x_n)\right\}
		$$
		with $\uptau(0)=0$. Thus, we can write
		$$
		D^*:=D\cap\mathcal{U}=\left\{(x_1,x_2,\dots,x_n)\in\mathcal{U}: x_1+\uptau(x_2,x_3,\dots,x_n)>0\right\}.
		$$
		Here, one must note that the equation of the tangent plane is given by
		\begin{align*}
			x_1&=-\uptau(0)+\nabla (-\uptau)(0)\cdot x,\\
			\text{i.e., } x_1&=-\nabla (\uptau)(0)\cdot x,
		\end{align*}
		since $\uptau(0)=0$. Now, as per our assumption $x_1=0$ is the tangent plane at $0$, we must have $\nabla\uptau(0)=0$. Also, after the elementary operations, the normal vector at the origin must be $(1,0,\dots,0)$. As $\nabla(r(x))|_{x=0}=(1,0,\dots,0)$, where $r(x):=x_1+\uptau(x_2,x_3,\dots,x_n)$, represents the normal vector at the origin, we must have $\nabla\uptau(0)=0$. 
		
		We now define the boundary flattening map $\uptheta$ near the origin as
		$$
		\uptheta(x)=(r(x),x_2,\dots,x_n)
		$$
		and show that for any arbitrary constant $c>1$, we can shrink the neighbourhood $\mathcal{U}$ in which $\uptheta$ is a $c$-bi-Lipschitz function. Indeed, one can calculate 
		\begin{align*}
			\uptheta(x)-\uptheta(y)
			&=(r(x),x_2,\dots,x_n)-(r(y),y_2,\dots,y_n)\\
			&=(x_1-y_1,\cdots,x_n-y_n)+(\uptau(x_2,x_3,\dots,x_n)-\uptau(y_2,y_3,\dots,y_n),0,\cdots,0).
		\end{align*}
		Therefore, we have
		\begin{align*}
			|\uptheta(x)-\uptheta(y)|
			&\leq |x-y|+|\uptau(x_2,x_3,\dots,x_n)-\uptau(y_2,y_3,\dots,y_n)|\\
			&\leq |x-y|+\left(\sup_{z\in\mathcal{U}}|\nabla\uptau(z)|\right)|(x_2,x_3,\dots,x_n)-(y_2,y_3,\dots,y_n)|,
		\end{align*}
		by the mean value theorem. Since we have $\nabla\uptau(0)=0$ and $\mathcal{U}$ is a neighbourhood of the origin, if we shrink down $\mathcal{U}$, the continuity of  $\nabla\uptau$ gives that $\sup_{z\in\mathcal{U}}|\nabla\uptau(z)|$ gets closer to $0$. Thus, for arbitrary small $\varepsilon>0$, we have 
		$$
		\sup_{z\in\mathcal{U}}|\nabla\uptau(z)|\leq \varepsilon.
		$$
		Hence, we write
		\begin{align*}
			|\uptheta(x)-\uptheta(y)|\leq |x-y|&+\varepsilon |(x_2,x_3,\dots,x_n)-(y_2,y_3,\dots,y_n)|\leq (1+\varepsilon)|x-y|,\\
			\text{i.e., } &|\uptheta(x)-\uptheta(y)|\leq c|x-y|
		\end{align*}
		and similarly, $|\uptheta(x)-\uptheta(y)|\geq c^{-1}|x-y|$ with $1+\varepsilon\leq c$ and $1-\varepsilon\geq 1/c$.
		
		Now, we apply the fact that the metric $m_D$ matches the function $\psi_D$ when $D$ is a half-plane. We denote $\mathbb{H}_1:=\left\{x\in\mathbb{R}^n:x_1>0\right\}$ and estimate
		\begin{align*}
			m_{\mathbb{H}_1}(\uptheta(x),\uptheta(y))=2\sinh^{-1}\frac{|\uptheta(x)-\uptheta(y)|}{2\sqrt{\delta'_D(x)\delta'_D(y)}},
		\end{align*}
		where $\delta'_D(z)$ represent the distance from the boundary after the transformation by $\uptheta$ which satisfies the relation $c^{-1}\delta_D(z)\leq \delta'_D(z)\leq c \delta_D(z)$. Therefore, we obtain
		\begin{align*}
			m_{\mathbb{H}_1}(\uptheta(x),\uptheta(y))
			\leq2\sinh^{-1}\frac{c^2|x-y|}{2\sqrt{\delta_D(x)\delta_D(y)}}
			&=2\sinh^{-1}\left(\frac{c^2d(D)|x-y|}{2\sqrt{\eta_D(x)\eta_D(y)}}\sqrt{\left(1-\frac{\delta_D(x)}{d(D)}\right)\left(1-\frac{\delta_D(y)}{d(D)}\right)}\right)\\
			&\leq 2\sinh^{-1}\left(\frac{c^2d(D)|x-y|}{2\sqrt{\eta_D(x)\eta_D(y)}}\right)\\
			&\leq c^2 \psi_D(x,y),
		\end{align*}
		where the second last inequality follows from the facts that the last factor inside the square root is less than $1$ and that $\sinh^{-1}(w)$ is an increasing function, and Lemma \ref{lem:sinh inequality} provides the final inequality. 
		
		To prove the other side of the inequality, we omit a few similar steps and start with 
		\begin{align*}
			m_{\mathbb{H}_1}(\uptheta(x),\uptheta(y))
			\geq 2\sinh^{-1}\left(\frac{c^{-2}d(D)|x-y|}{2\sqrt{\eta_D(x)\eta_D(y)}}\sqrt{\left(1-\frac{\delta_D(x)}{d(D)}\right)\left(1-\frac{\delta_D(y)}{d(D)}\right)}\right).
		\end{align*}
		By observing that we are in a shrinking neighbourhood $\mathcal{U}$, one can choose $\varepsilon_1>0$ such that $\delta_D(z)/d(D)<\varepsilon_1$, for all $z\in \mathcal{U}$. Hence, the manipulations continue as
		\begin{align*}
			m_{\mathbb{H}_1}(\uptheta(x),\uptheta(y))
			&\geq 2\sinh^{-1}\left(c^{-2}(1-\varepsilon_1)\frac{d(D)|x-y|}{2\sqrt{\eta_D(x)\eta_D(y)}}\right)\\
			&\geq c^{-2}(1-\varepsilon_1)\psi_D(x,y),
		\end{align*}
		where the last inequality follows from Lemma \ref{lem:sinh inequality}. Hence, the final two-sided inequality obtained is 
		\begin{align}\label{two sided inequality of psi and mm in H}
			c^{-2}(1-\varepsilon_1)\psi_D(x,y)
			\leq m_{\mathbb{H}_1}(\uptheta(x),\uptheta(y))
			\leq c^2 \psi_D(x,y).
		\end{align}
		Now, for any point $z\in D^*$, we use the shrinking neighbourhood idea to see
		\begin{align*}
			\frac{d(D^*)}{\eta_{D^*}(z)}
			=\frac{1}{\delta_{D^*}(z)\left(1-\frac{\delta_{D^*}(z)}{d(D^*)}\right)}
			<\frac{1}{1-\varepsilon_1}\frac{1}{\delta_{D^*}(z)}
			\leq \frac{1}{1-\varepsilon_1}\frac{d(D)}{\eta_{D}(z)},
		\end{align*}
		which further implies that for any two points in the shrinking neighbourhood $x,y$, we have
		\begin{align}\label{eq:m inequality2}
			m_{D}(x,y)\leq m_{D^*}(x,y)\leq \frac{1}{1-\varepsilon_1} m_{D}(x,y).
		\end{align}
		Further, the use of Proposition \ref{prop:distortion-Lip} and equation \eqref{eq:m inequality2} gives
		\begin{align*}
			c^{-3}m_D(x,y)\leq c^{-3}m_{D^*}(x,y)\leq m_{\uptheta (D^*)}(\uptheta(x),\uptheta(y))\leq c^3m_{D^*}(x,y)\leq \frac{c^3}{1-\varepsilon_1}m_D(x,y).
		\end{align*}
		Let $\gamma$ be a geodesic joining $\uptheta(x)$ to $\uptheta(y)$ in $\mathbb{H}_1$. Then, using a similar argument as above, we calculate
		\begin{align*}
			m_{\uptheta (D^*)}(\uptheta(x),\uptheta(y))
			\leq \int_{\gamma}\frac{d(\uptheta (D^*))}{\eta_{\uptheta (D^*)}(z)}|dz|
			\leq \frac{1}{1-\varepsilon_2}m_{\mathbb{H}_1}(\uptheta(x),\uptheta(y)),
		\end{align*}  
		and it follows from the domain monotonicity property that 
		$$
		m_{\mathbb{H}_1}(\uptheta(x),\uptheta(y))\leq m_{\uptheta (D^*)}(\uptheta(x),\uptheta(y)).
		$$
		Hence, we obtain
		\begin{align*}
			c^{-3}m_{\mathbb{H}_1}(\uptheta(x),\uptheta(y))
			\leq m_D(x,y)\leq \frac{c^3}{1-\varepsilon_2}m_{\mathbb{H}_1}(\uptheta(x),\uptheta(y))
		\end{align*}
		and combining this with \eqref{two sided inequality of psi and mm in H} we get
		\begin{align*}
			\frac{c^{-5}}{1-\varepsilon_1}\leq\frac{m_D(x,y)}{\psi_D(x,y)}\leq \frac{c^{5}}{1-\varepsilon_2}.
		\end{align*}
		Finally, letting $\varepsilon_1,\varepsilon_2\to 0$ and $c\to 1^+$ completes the proof.
	\end{proof}
	The following result is an immediate consequence of Theorem \ref{thm:estimate near bdry}.
	\begin{corollary}
		For a bounded domain $D$ in $\mathbb{R}^n$ with $\mathcal{C}^1$-smooth boundary, we have 
		\begin{align}
			\lim_{\substack{x\to \partial D}}\frac{m_D(x,y)}{\psi_D(x,y)}=1 \text{ uniformly in } x\neq y.
		\end{align}
	\end{corollary}
	
\subsection{A generalized metric as an improved lower bound of $m_D$}
In \cite{NiAn}, Nikolov--Andreev defined a metric $i_D$ in a domain $D$ of $\mathbb{R}^n$ as
\begin{equation}
	i_D(x,y) :=2\log\frac{\delta_{D}(x)+\delta_{D}(y)+|x-y|}{2\sqrt{\delta_{D}(x)\delta_{D}(y)}}
\end{equation}	
and used it to estimate the quasihyperbolic metric $k_D$ and connect it to the work of Dovgoshey et al.~\cite{DHV}. Later, this metric has been further investigated in many other works by different authors \cite{Moc25, LRZ, ZZH}. Now, we show that the quantity $i_D^*$ defined in \eqref{eq:i*}, generalizing $i_D$, is a metric.
    \begin{proposition}\label{prop:i is a metric}
    	The expression given by 
    	\begin{equation*}
    		i_D^*(x,y) :=2\log\frac{\eta_{D}(x)+\eta_{D}(y)+d(D)|x-y|}{2\sqrt{\eta_{D}(x)\eta_{D}(y)}}
    	\end{equation*}	
    	defines a metric in a domain $D$ in $\mathbb{R}^n$.
    \end{proposition}
    \begin{proof}
    	We mainly show that vanishing metric value implies that the considered points are the same and that the triangle inequality holds, as symmetry and positivity are immediate. 
    	
    	Suppose for a pair of points $x,y\in D$, $i_D^*(x,y)=0$, which is equivalent of saying
    	\begin{align*}
    		\eta_{D}(x)+\eta_{D}(y)+d(D)|x-y|=2\sqrt{\eta_{D}(x)\eta_{D}(y)}\leq \eta_{D}(x)+\eta_{D}(y).
    	\end{align*}
    	Note that we have used the AM-GM inequality in the last step. Thus, we obtain
    	\begin{align*}
    		|x-y|\leq 0.
    	\end{align*}
    	The non-negativity of Euclidean distance implies $x=y$. 
    	
    	Now for any three points $x,y,z\in D$, the triangle inequality for the metric $i_D^*$ is equivalent to
    	\begin{align*}
    		\eta_{D}(x)+\eta_{D}(y)+d(D)|x-y|\leq \frac{1}{2\eta_{D}(z)}(\eta_{D}(x)+\eta_{D}(z)+&d(D)|x-z|)\\
    		&(\eta_{D}(z)+\eta_{D}(y)+d(D)|z-y|).
    	\end{align*}
    	From the Euclidean triangle inequality, it is sufficient to prove that 
    	\begin{align*}
    		2\eta_D(z)[\eta_{D}(x)+\eta_{D}(y)+d(D)|x-z|+d(D)|z-y|]\leq (\eta_{D}(x)+&\eta_{D}(z)+d(D)|x-z|)\\
    		&(\eta_{D}(z)+\eta_{D}(y)+d(D)|z-y|).
    	\end{align*}
    	Further simplifications gives
    	\begin{align*}
    		-(\eta_D(z)-\eta_D(x)-d(D)|x-z|)(\eta_D(z)-\eta_D(y)-d(D)|z-y|)\leq 0,
    	\end{align*}
    	which is always true due to Lemma \ref{lemma:eta is lips}.
    \end{proof}
    \begin{corollary}
    	Let $f:D\to \mathbb{R}$ be a positive $\alpha$-Lipschitz function on a domain $D$ in $\mathbb{R}^n$ with $\alpha>0$. Then the expression given by 
    	\begin{equation*}
    		i_{D,f}^*(x,y) :=2\log\frac{f(x)+f(y)+\alpha|x-y|}{2\sqrt{f(x)f(y)}}
    	\end{equation*}	
    	defines a metric.
    \end{corollary}  
    \begin{proof}
    	The proof goes in the same way as the proof of Proposition \ref{prop:i is a metric}.
    \end{proof} 
    Now, we state our main result of this subsection and a consequence of it. The proof of Theorem \ref{thm:improved lower bound for mD} follows from the next two propositions proved afterwards.
      \begin{theorem}\label{thm:improved lower bound for mD}
    	Let $D$ be a bounded domain in $\mathbb{R}^n$. Then for any pair of points $x,y\in D$
    	\begin{equation*}
    		\zeta_D(x,y)\leq i_D^*(x,y)\leq m_D(x,y).
    	\end{equation*}
    \end{theorem}
    \begin{corollary}
    	Let $D$ be an unbounded domain in $\mathbb{R}^n$. Then for any pair of points $x,y\in D$
    	\begin{equation*}
    		j_D(x,y)\leq i_D(x,y)\leq k_D(x,y).
    	\end{equation*}
    \end{corollary} 
    \begin{proof}
    	The proof is a direct consequence of the unboundedness of the domain $D$.
    \end{proof}
    \begin{proposition}\label{prop:comparison_i*&zeta}
    	Let $D$ be a bounded domain in $\mathbb{R}^n$. Then for any pair of points $x,y\in D$
    	\begin{equation*}
    		\zeta_D(x,y)\leq i_D^*(x,y)\leq 2\zeta_D(x,y).
    	\end{equation*}
    	Moreover, both inequalities are sharp.
    \end{proposition}
    \begin{proof}
    	Let $x,y\in D$ be two arbitrary points in $D$. Without loss of generality we assume that $\eta_D(x)\leq \eta_D(y)$. The $d(D)$-Lipschitz property of the function $\eta_D$ implies that
    	\begin{align*}
    		(\eta_D(x)-\eta_D(y)+d(D)|x-y|)^2\geq 0.
    	\end{align*}
    	Expansion of the square term and a small rearrangement gives
    	\begin{align*}
    		(\eta_D(x)+\eta_D(y)+d(D)|x-y|)^2\geq& 4\eta_D(x)\eta_D(y)+4\eta_D(x)d(D)|x-y|,\\
    		\text{i.e., } \frac{(\eta_D(x)+\eta_D(y)+d(D)|x-y|)^2}{4\eta_D(y)}&\geq \eta_D(x)+d(D)|x-y|.
    	\end{align*}
    	Finally, we write the above expression as
    	\begin{align*}
    		\frac{(\eta_D(x)+\eta_D(y)+d(D)|x-y|)^2}{4\eta_D(x)\eta_D(y)}\geq\frac{\eta_D(x)+d(D)|x-y|}{\eta_D(x)},
    	\end{align*}
    	and taking logarithm on both the sides establishes the first inequality. 
    	
    	To see the second inequality, we begin with
    	\begin{align*}
    	\frac{(\eta_{D}(x)+\eta_{D}(y)+d(D)|x-y|)^2}{4\eta_{D}(x)\eta_{D}(y)}
    		=&1+\frac{d(D)|x-y|}{2\sqrt{\eta_D(x)\eta_D(y)}}+\left(\frac{d(D)-(\eta_D(y)-\eta_D(x))}{2\eta_D(x)\eta_D(y)}\right)^2\\
    		\leq& \left(1+\frac{d(D)|x-y|}{\sqrt{\eta_D(x)\eta_D(y)}}\right)^2,
    	\end{align*}
    	and taking the logarithm on both sides, the second inequality is obtained. 
    	
    	To see that the inequalities are sharp, consider the pair of diametrically opposite points $x$ and $-x$ on the real diameter of the unit disk $\mathbb{D}$. A direct computation shows that 
    	\begin{align*}
    		\lim_{x\to 0}\frac{i_D^*(x,-x)}{\zeta_{\mathbb{D}}(x,-x)}=1 \quad \text{and}\quad \lim_{x\to 1}\frac{i_D^*(x,-x)}{\zeta_{\mathbb{D}}(x,-x)}=2.
    	\end{align*}
    	Hence, both the constant in the ineualities are best possible.
    \end{proof}
    \begin{corollary}
    	Let $D$ be an unbounded domain in $\mathbb{R}^n$. Then for any pair of points $x,y\in D$
    	\begin{equation*}
    		j_D(x,y)\leq i_D(x,y)\leq 2j_D(x,y).
    	\end{equation*}
    \end{corollary}
    \begin{proof}
    	The unboundedness of the domain $D$ yields the result.
    \end{proof}
    \begin{proposition}\label{prop:m_D is inner metric of i*}
    	For any two pair of points $x,y\in D$, we have 
    	\begin{equation*}
    		\begin{aligned}
    			2\log\left(1 + \frac{d(D)|x-y|}{2\sqrt{\eta_D(x)\eta_D(y)}}\right) \leq i_D^*(x,y) \leq& 2\log\left(1 + \frac{d(D)|x-y|}{2\sqrt{\eta_D(x)\eta_D(y)}} \right. \\
    			&\left. + \frac{d(D)^2|x-y|^2}{2\sqrt{\eta_D(x)\eta_D(y)}(\sqrt{\eta_D(x)}+\sqrt{\eta_D(y)})^2}\right).
    		\end{aligned}
    	\end{equation*}
    	In particular, the following limit holds true: 
    	\begin{equation*}
    		\lim_{y\to x}\frac{i_D^*(x,y)}{|x-y|}=\frac{d(D)}{\eta_{D}(x)},
    	\end{equation*}
    	and this shows that the metric $m_D$ acts as an inner metric of $i_D^*$.
    \end{proposition}
    \begin{proof}
    	We start the proof by estimating the argument inside the logarithm in the definition of $i_D^*$ as follows:
    	\begin{align}\label{eq:main eq in 2sided inequality of i*}
    		\frac{\eta_{D}(x)+\eta_{D}(y)+d(D)|x-y|}{2\sqrt{\eta_{D}(x)\eta_{D}(y)}}
    		=&\frac{\left(\sqrt{\eta_D(x)}-\sqrt{\eta_D(y)}\right)^2+2\sqrt{\eta_D(x)\eta_D(y)}+d(D)|x-y|}{2\sqrt{\eta_D(x)\eta_D(y)}}\nonumber \\ 
    		=&1+\frac{d(D)|x-y|}{2\sqrt{\eta_D(x)\eta_D(y)}}+\frac{\left(\sqrt{\eta_D(x)}-\sqrt{\eta_D(y)}\right)^2}{2\sqrt{\eta_D(x)\eta_D(y)}}
    	\end{align}
Dropping the non-negative squared term yields the lower bound.
    	
    	For the second inequality, we use
    	\begin{align}\label{eq:inequality for eta form Lip property}
    		\sqrt{\eta_D(x)}-\sqrt{\eta_D(y)}\leq \frac{d(D)|x-y|}{\sqrt{\eta_D(x)}+\sqrt{\eta_D(y)}},
    	\end{align}
    	which is a consequence of the fact that the function $\eta_D$ is $d(D)$-Lipschitz. Estimating the last term in \eqref{eq:main eq in 2sided inequality of i*} using \eqref{eq:inequality for eta form Lip property}, we deduce the result.
    	
    	Taking the limit as $y \to x$ after dividing by $|x-y|$, the sandwich rule of limits yields the local density to be $d(D)/\eta_D(x)$, establishing that $m_D$ is the inner metric of $i_D^*$.
    \end{proof}
    
    \subsection*{Concluding remarks} In \cite[Proposition 5.5]{MaNaSa 2}, the authors showed that the metric $\zeta_D$ has $m_D$ as its inner metric. Moreover, in Proposition \ref{prop:m_D is inner metric of i*}, we showed that the metric $m_D$ is the inner metric of $i_D^*$. Hence, two distinct metrics can have the same metric as their inner metric.
    
	\section*{\bf Acknowledgements}
The author is financially supported by the University Grants Commission (UGC) with Ref. No.: 221610091493.

\end{document}